\documentclass[11pt,reqno,a4paper]{amsart}
\usepackage{hyperref, soul}
\usepackage{amsmath}
\usepackage{amssymb,a4wide}
\usepackage{amsthm}
\usepackage[margin=3cm]{geometry}
\usepackage{graphicx,epstopdf}
\usepackage{enumerate}

\hypersetup{
    colorlinks=true,
    linkcolor=black,
    filecolor=black,      
    urlcolor=blue,
    citecolor=red,
}

\newcommand{\guio}[1]{\nobreakdash-\hspace{0pt}}

\numberwithin{equation}{section}
\newtheorem{theorem}{Theorem}[section]
\newtheorem{lemma}[theorem]{Lemma}
\newtheorem{prop}[theorem]{Proposition}

\theoremstyle{definition}

\newtheorem{remark}[theorem]{Remark}

\newcommand{\R}{\mathbb{R}}
\newcommand{\C}{\mathbb{C}}
\newcommand{\Z}{\mathbb{Z}}
\renewcommand{\S}{\mathbb{S}}

\newcommand{\ep}{\varepsilon}

\newcommand{\supp}{\operatorname{supp}}

\newcommand{\Rd}{r_d}

\title[Explicit minimisers for anisotropic Riesz energies]
{Explicit minimisers for anisotropic \\ Riesz energies
}

\author[R.L.~Frank]{R.L.~FRank}
\author[J.~Mateu]{J.~Mateu}
\author[M.G.~Mora]{M.G.~Mora}
\author[L.~Rondi]{L.~Rondi}
\author[L.~Scardia]{L.~Scardia}
\author[J.~Verdera]{J.~Verdera}

\address[R.L. Frank]{Mathematisches Institut, Ludwig-Maximilians Universit\"at M\"unchen, Germany, and Munich Center for Quantum Science and Technology, M\"unchen, Germany, and Department of Mathematics, Caltech, Pasadena, CA, USA}
\email{r.frank@lmu.de}

\address[J. Mateu]{Department de Matem\`atiques, Universitat Aut\`onoma de Barcelona, and Centre de Recerca Matem\`atica, Catalonia}
\email{Joan.Mateu@uab.cat}

\address[M.G. Mora]{Dipartimento di Matematica, Universit\`a di Pavia, Italy}
\email{mariagiovanna.mora@unipv.it}

\address[L. Rondi]{Dipartimento di Matematica, Universit\`a di Pavia, Italy}
\email{luca.rondi@unipv.it}

\address[L. Scardia]{Department of Mathematics, Heriot-Watt University, United Kingdom}
\email{L.Scardia@hw.ac.uk}

\address[J. Verdera]{Department de Matem\`atiques, Universitat Aut\`onoma de Barcelona, and Centre de Recerca Matem\`atica, Catalonia}
\email{Joan.Verdera@uab.cat}

\date{}

\begin{document}

\begin{abstract}

In this paper we characterise the energy minimisers of a class of nonlocal interaction energies where the attraction is quadratic, and the repulsion is Riesz-like and anisotropic. In particular we show that, if the Fourier transform of the repulsive potential is positive, the minimiser is supported on a fully-dimensional ellipsoid, and its density is given by a Barenblatt-type profile. Our technique of proof is based on a Fourier representation of the potential of such measures, that extends a previous formula established by some of the authors in the Coulomb case.
\bigskip


\medskip

\end{abstract}

\maketitle

\section{Introduction}

Motivated by applications in physics, biology and economics, in this paper we study a mean-field model of particles (or agents) interacting through a repulsive, anisotropic Riesz potential in a quadratic confinement. More precisely, we consider nonlocal energies of the form
\begin{equation}\label{en:general}
I(\mu):= \int_{\R^d} (W\ast\mu)(x) \,d\mu(x) +\int_{\R^d}|x|^2 \,d\mu(x),
\end{equation}
defined on probability measures $\mu\in\mathcal{P}(\R^d)$, for $d\geq 2$. In \eqref{en:general} the potential $W$ is of the form 
\begin{equation}\label{potdef}
W(x):=\frac1{|x|^s}\Psi\left(\frac{x}{|x|}\right),\quad s\in (0,d),
\end{equation}
for $x\neq 0$ and $W(0)=+\infty$, and the profile $\Psi$ is
even, strictly positive on $\mathbb{S}^{d-1}$, and such that both $W$ and $\widehat{W}$ are continuous on $\mathbb{S}^{d-1}$. 
Here $\widehat W$ denotes the Fourier transform of $W$, see Section~\ref{FTW:sect}.
Explicit examples of anisotropic potentials of the form above can be found in \cite{CMMRSV2,Giorgio,MRS} in the case $s=d-2$, and are used to model interactions of defects in metals, see \cite{HL}.
We also observe that small perturbations of the isotropic Riesz potential - namely, kernels of the form \eqref{potdef} with a profile $\Psi=1+\ep\Phi$, where $\Phi$ is even and smooth and $\ep$ is sufficiently small - automatically satisfy our assumptions (see \cite{MMRSV-st}).

Alternatively, one could consider the fully nonlocal analogue of \eqref{en:general},
$$
\widetilde I(\mu) := \int_{\R^d} (\widetilde W\ast\mu)(x)\,d\mu(x),
$$
where the attractive/repulsive interaction potential $\widetilde W$ is given by $\widetilde W(x):=W(x) +\frac12|x|^2$. In fact, by translation invariance one can reduce the minimisation of $\widetilde I$ to probability measures $\mu$
with compact support and with $\int_{\R^d} x\,d\mu(x)=0$, and for such measures the two functionals coincide.
In what follows, we find it more convenient to work with $I$ rather than $\widetilde I$ and will focus on the formulation \eqref{en:general}.

We observe that for a constant profile $\Psi$, which without loss of generality can be assumed to be $\Psi\equiv 1$, the potential $W$ reduces to the classical (radially symmetric) Riesz potential. In this case it was proved in \cite{CaSh1} that for $s\in (\max\{d-4,0\},d)$ the unique minimiser of $I$ over $\mathcal{P}(\R^d)$ is given by the Barenblatt measure
\begin{equation}\label{intro:explicit}
\mu_{\text{iso},d}(x)= c_d\,(1-|x/r_d|^2)^{\frac{s+2-d}2}\chi_{\Rd \overline{B}}(x),
\end{equation}
where $c_d>0$ and $r_d>0$ are explicit constants depending on $d$, $s$, $B=B_1(0)$ is the unit ball, $\chi_{\Rd \overline{B}}$ denotes the characteristic function of $\Rd \overline{B}$, and where we identified the measure with its density $\mu_{\text{iso},d}\in L^1(\R^d)$ (see also \cite{CaVa}). For $0<s\leq d-4$, instead, it was proved in  \cite{FrankMatzke2023} that the unique minimiser is given by the uniform probability measure on a sphere whose radius depends on $d$.
Note that in some of these previous works the equivalent minimisation problem for $\widetilde I$ was considered. The values of the constants $c_d$ and $r_d$ can be extracted from \cite{FrankMatzke2023}. For a related result in dimension $d=1$, see \cite{Fr0}.

In the present paper we are interested in the \emph{anisotropic case} where $\Psi$ is not necessarily constant. We show that for dimension $d=3$ and for the full range $s\in (0,3)$, if $\widehat W>0$, the unique minimiser $\mu_{\text{min}}$ of the energy $I$ over $\mathcal{P}(\R^d)$ is supported on a fully-dimensional ellipsoid $E\subset \R^3$, and its density is given by a Barenblatt-type profile.
Roughly speaking, as long as $\widehat W>0$, the effect of any anisotropic potential of the form \eqref{potdef} on the minimiser is simply that of `deforming' the support of \eqref{intro:explicit}
from a ball into a suitable ellipsoid. Our proof extends to higher dimensions for a
partial range of Riesz homogeneities.
To be precise, our main result reads as follows.

\begin{theorem}\label{chara}
Let $s\in [d-3,d)\cap (0,5]$, and let $W$ be as in \eqref{potdef} with $\Psi$ even. Assume that $W$ and $\widehat{W}$  are strictly positive and continuous on $\mathbb{S}^{d-1}$. Then there exists a unique minimiser $\mu_0$ of $I$ over $\mathcal{P}(\R^d)$. It is given by the push-forward of the measure \eqref{intro:explicit} through the function $x\mapsto RD(a/r_d)x$, for some rotation $R\in SO(d)$ and 
some $a\in \R^d$ with $a_i=a\cdot e_i>0$. More precisely,
\begin{equation}\label{thm:min}
\mu_0(x)=  \frac{c_d}{\prod_{j=1}^d (a_j/r_d)}\left(1-\left|D\big(\tfrac1a\big)R^Tx\right|^2\right)^{\frac{s+2-d}2}\chi_{ E}(x),
\end{equation}
where $E$ is the ellipsoid $R D(a) \overline{B}$, and $c_d$ and $r_d$ are the constants from \eqref{intro:explicit}.
\end{theorem}

In the statement above $D(a)$ denotes the diagonal matrix such that $(D(a))_{ii}=a_i$, and $D\big(\tfrac1a\big)$ denotes the diagonal matrix such that $(D(\tfrac 1a))_{ii}=1/a_i$.

In the two-dimensional case the result of Theorem~\ref{chara} has been proved by Carrillo and Shu in \cite{CS2d}. They also considered the three-dimensional case in  \cite{CS3d}, but only under additional symmetry assumptions on the potential, which essentially made the problem two-dimensional. In \cite{MMRSV-3d} a new strategy of proof was developed, which allowed to successfully remove the additional assumptions in \cite{CS3d} in the special case of Coulomb singularity $s=d-2$ in three dimensions (corresponding to $s=1$). In \cite{Mora-Muenster} this method was adapted to the case of logarithmic interactions in two dimensions, which correspond loosely speaking to `$s=0$'.

In the paper \cite{CS3d} Carrillo and Shu raised the question of whether the methods of \cite{MMRSV-3d, Mora-Muenster} could be extended, in the three-dimensional case, to the full range $s\in (0,3)$ of homogeneity, beyond the Coulomb singularity $s=1$. In this paper we give a positive answer to that question, and characterise the minimisers in the full range $s\in (0,3)$, without the additional symmetry assumptions on $\Psi$ required in \cite{CS3d}. Moreover, since the range $[d-3,d)\cap(0,5]$ considered here includes the Coulomb exponent $s=d-2$ for $3\leq d\leq 7$, our result extends those of \cite{MMRSV-3d, Mora-Muenster} to dimension $d\leq 7$.

The strategy of proof of our main result consists in showing that there exist a rotation $R\in SO(d)$ and a diagonal matrix $D(a)$ with positive entries such that the candidate 
$\mu_{\text{min}}$, that is, the push-forward of $\mu_{\text{iso},d}$ through the function $x\mapsto RD(a/r_d)x$, satisfies the Euler-Lagrange conditions 
\begin{align}\label{EL1-intro}
\left(W\ast \mu_{\text{min}}  \right) (x)+ \frac{1}{2}|x|^2&= C \quad \text{for } x \in  E,\\
\left(W\ast \mu_{\text{min}}  \right) (x)+ \frac{1}{2}|x|^2&\ge C \quad \text{for }x \in \R^d, \label{EL2-intro}
\end{align}
where $C$ is a constant, and $E$ is the ellipsoid $R D(a)\overline{B}$.
To this aim we follow the methodology developed in our paper \cite{MMRSV-3d} which requires, as a first step, to write the potential $W\ast \mu_{\text{min}}$ in Fourier space. This computation is considerably more involved than in the Coulomb case, since the Fourier transform of $\mu_{\text{min}}$ involves Bessel and hypergeometric functions, and only works for  $s\in (d-4,d)\cap(0,5]$, where the further restriction $s\in (0,5]$ is needed to deal with the lack of integrability of $\widehat{W\ast \mu_{\text{min}}}$. In the isotropic case, similar computations can be found in \cite{BIK2015-ARMA}, for the derivation of self-similar profiles of the nonlocal porous medium equation, and in \cite{FrankMatzke2023}, for the characterisation of minimisers of attractive-repulsive energies, again in the isotropic framework. Once \eqref{EL1-intro}--\eqref{EL2-intro} are written in Fourier terms, we proceed as follows.

For the proof of \eqref{EL1-intro} we introduce a family $W_t$ of potentials interpolating between the isotropic case (corresponding to $t=0$) and $W$ in \eqref{potdef} (corresponding to $t=1$), with $\widehat{W_t}>0$ for every $t\in [0,1]$. We then resort to a continuity argument on the parameter $t$ to show that the set $T\subset [0,1]$ where equation \eqref{EL1-intro}, with $W$ replaced by $W_t$, is satisfied is nonempty, open and closed in $[0,1]$. This implies that $T=[0,1]$ and concludes the proof of \eqref{EL1-intro}. It is at this step of the proof that the restriction $s\geq d-3$ becomes necessary. It is not clear whether this is a technical condition due to the argument of proof or in fact the existence of an ellipsoid satisfying \eqref{EL1-intro} may fail for $s\in(d-4, d-3)$.
Finally we show that whenever a measure of the form \eqref{intro:explicit} is a solution of  \eqref{EL1-intro}, it also satisfies \eqref{EL2-intro}. While in \cite{MMRSV-3d} this step is immediate, care is needed in the present case, due to the extra parameter $s$ and to the more complicated nature of the candidate minimiser.

In Section~\ref{lossdimsec} we briefly discuss the degenerate case $\widehat{W}\geq 0$, and show that the energy minimiser may be supported on a lower-dimensional set. In particular, in dimension $d=3$, we have the following cases, depending on the strength of the singularity of the potential at the origin. For $s\in (0,1)$ the minimiser must be supported on a set of dimension at least one, and the support may be a segment, an ellipse, or an ellipsoid. For $s\in [1,2)$ the dimension of the support must be at least 2, so the segment is excluded. For $s\in [2,3)$ the minimiser is fully dimensional. Explicit examples, see, e.g., \cite[Example 3.4]{MMRSV-3d} for the case of Coulomb interactions $s=1$, show that the loss of dimensionality of the minimiser can in fact occur.  

\section{Preliminaries} 
In this section we collect some definitions and preliminary results that will be needed in the paper. We also establish some notation.

\subsection{Existence and uniqueness of a minimiser} 
In Proposition~\ref{ex+uniq} we prove existence and uniqueness of a minimiser of the energy $I$, and show that it is characterised by the Euler-Lagrange conditions for $I$. This is the first step of the proof of Theorem~\ref{chara}, and is quite straightforward. Section~\ref{sec:main-res} will be devoted to the proof of the main part of Theorem~\ref{chara}, where we show that there exists an ellipsoid $E$ such that the corresponding measure \eqref{thm:min} is the unique minimiser of the energy.

\begin{prop}\label{ex+uniq}
Let $s\in (0,d)$, and let $W$ be as in \eqref{potdef} with $\Psi$ even, strictly positive on $\mathbb{S}^{d-1}$, and such that $W$ and $\widehat{W}$ are continuous on $\mathbb{S}^{d-1}$. 
Assume that $\widehat{W}\geq 0$ on $\mathbb{S}^{d-1}$.
Then there exists a unique minimiser $\mu_0$ of $I$ over $\mathcal{P}(\R^d)$. It is characterised by the following Euler-Lagrange conditions:
\begin{align}\label{EL1}
\left(W\ast \mu_0  \right) (x)+ \frac{1}{2}|x|^2&= C \quad \text{for }\mu_0\text{-a.e. } x \in \supp\mu_0,\\
\left(W\ast \mu_0  \right) (x)+ \frac{1}{2}|x|^2&\ge C \quad \text{for }x \in \R^d \setminus N \text{ with } \operatorname{Cap}_{d-s}(N)=0, \label{EL2}
\end{align}
where $\supp\mu_0$ stands for the support of $\mu_0$, $C$ is a constant, and $\operatorname{Cap}_{d-s}$ is the $(d-s)$-Riesz capacity defined as in \cite{Land}.
\end{prop}

\begin{proof} The proof is by now standard. Indeed, the positivity of $\Psi$ implies that the energy $I$ is bounded from below by the quadratic confinement, and this guarantees that minimising sequences are tight. Existence of the minimisers then follows by the lower-semicontinuity of the energy. By the lower bound of the energy in terms of the confinement it also follows that any minimiser has a compact support. 

Uniqueness of the minimiser is a consequence of the strict convexity of the energy on the class of measures with compact support and finite interaction energy. This, in its turn, follows by the assumption that the Fourier transform of the potential is non-negative on the sphere. 
For further details see, e.g., \cite[Proposition~2.1]{CMMRSV2}. Note that the first inequality in \cite[eq.~(2.11)]{CMMRSV2} follows by superharmonicity in the case $s\in(0,d-2]$, whereas for $s\in(d-2,d)$ it can be obtained as in the proof of \cite[Theorem~1.11]{Land}.

The characterisation of the minimiser in terms of the Euler-Lagrange conditions is due to the strict convexity of the energy. The derivation of the Euler-Lagrange conditions \eqref{EL1}--\eqref{EL2} follows by a simple adaptation of the proof of \cite[Theorem 3.1]{MRS}.
\end{proof}

%%%%%%
\subsection{Properties of the Gamma function}\label{sect:propG}
The Gamma function $\Gamma(z):= \int_0^\infty e^{-t} t^{z-1}\,dt$ is defined and analytic in the region $\Re(z)>0$, where $\Re(z)$ denotes the real part of $z\in \C$. We collect below a number of useful properties:
\begin{enumerate}
\item $\Gamma(n)=(n-1)!$ for $n\in \mathbb{N}$;
\smallskip
\item $\Gamma(1+z)=z\,\Gamma(z)$ for every $z\in \mathbb{C}$, $z\neq0, -1, -2, \ldots$;
\smallskip
\item $\Gamma(1-z)\Gamma(z)=\frac{\pi}{\sin(\pi z)}$ for every $z\in \mathbb{C}\setminus \mathbb{Z}$ (Euler reflection formula);
\smallskip
\item $\Gamma(z)\Gamma(z+\tfrac12)=2^{1-2z}\sqrt{\pi}\, \Gamma(2z)$ for every $z\in \mathbb{C}$, $z\neq0, -\frac12, -1, -\frac32, \ldots$ (Legendre duplication formula).
\end{enumerate}
For more details we refer to \cite[Sections~1.1 and~1.2]{Lebedev}.

%%%%%%%%%%%%%

\subsection{The Fourier transform.}\label{FTW:sect}
The definition of the Fourier transform we adopt in this paper is 
$$
\widehat{f}(\xi) =\frac{1}{(2\pi)^{d/2}}\int_{\R^d} f(x) e^{-i \xi \cdot x}\, dx, \quad \xi \in \R^d,
$$
for functions $f$ in the Schwartz space $\mathcal{S}$.
Correspondingly, the inverse Fourier transform is defined as
$$
f(x)=\frac{1}{(2\pi)^{d/2}} \int_{\R^d} \widehat{f}(\xi) e^{i \xi \cdot x}\, d\xi, \quad x \in \R^d.
$$

\subsection{The Bessel function of first kind}\label{Besselsec} The Bessel function of the first kind and arbitrary order $\nu\geq 0$ is defined, for $0\leq x<+\infty$, in terms of the following series
$$
J_\nu(x) = \sum_{n=0}^\infty \frac{(-1)^n (x/2)^{\nu+2n}}{\Gamma(n+1)\Gamma(n+\nu+1)},
$$
and is a real and bounded function. Since for fixed $x\geq 0$ the terms of the series are analytic functions of the variable $\nu$ (by the analyticity of the Gamma function), the uniform convergence of the series guarantees that $J_\nu$ is also an analytic function of $\nu$ (see, e.g. \cite[Section~5.3]{Lebedev}).

The behaviour of $J_\nu$ for small and large values of $x$ is described by the asymptotic formulas
\begin{align*}
J_\nu(x)&\sim \frac{x^\nu}{2^\nu \Gamma(1+\nu)}, \quad \text{as } x\to 0^+,\\
J_\nu(x)&\sim \sqrt{\frac{2}{\pi x}}\cos\left(x-\frac12 \nu\pi - \frac14\pi\right), \quad \text{as } x\to +\infty.
\end{align*}
For more details we refer to \cite[Section~5.16]{Lebedev}.

\subsection{The hypergeometric function} By the hypergeometric series is meant the power series
$$
\sum_{n=0}^\infty \frac{(a)_n(b)_n}{(c)_n}\frac{z^n}{n!}, 
$$
where $z\in \mathbb{C}$, $a,b,c \in \mathbb{C}$, with $c\notin \{0, -1, -2, \dots\}$, and the symbol $(\gamma)_n$ denotes the quantity
$$
(\gamma)_0=1, \quad (\gamma)_n= \gamma(\gamma+1)\cdot\dots\cdot(\gamma+n-1), \quad n=1,2,\dots.
$$
We note that for $\gamma\neq 0,-1,-2,\ldots,$ we have $(\gamma)_n= \Gamma(\gamma+n)/\Gamma(\gamma)$.
In the special case where either $a$ or $b$ is a non-positive integer, the series has a finite number of terms, and its sum reduces to a polynomial in $z$. In particular, if $b=-m$, with $m$ a non-negative integer, then 
\begin{equation}\label{poli}
\sum_{n=0}^\infty \frac{(a)_n(-m)_n}{(c)_n}\frac{z^n}{n!} = 
\sum_{n=0}^m (-1)^n \binom{m}{n} \frac{(a)_n}{(c)_n}z^n.
\end{equation}
In the general case the series converges for $|z|<1$, the sum of the series is denoted by 
\begin{equation}\label{2F1}
{}_2F_1(a,b;c;z) = \sum_{n=0}^\infty \frac{(a)_n(b)_n}{(c)_n}\frac{z^n}{n!}, \quad |z|<1,
\end{equation}
and is called the hypergeometric function. For fixed $z\in \mathbb{C}$ with $|z|<1$, ${}_2F_1$ is an entire function of $a$ and $b$ and a meromorphic function of $c$, with simple poles at the non-positive integers.

One can see easily that ${}_2F_1$ is invariant under the permutation of its first two arguments, and that 
\begin{equation}\label{der:hyp}
\frac{d}{dz}\,{}_2F_1(a,b;c;z)=\frac{ab}{c}\,{}_2F_1(a+1,b+1;c+1;z).
\end{equation}
It is easy to deduce from \eqref{der:hyp} that
\begin{equation}\label{der:hyp2}
	\frac{d}{dz}\left( z^{-a} {}_2F_1(a,b;c;z^{-1}) \right) = - a z^{-a-1} \,{}_2F_1(a+1,b;c;z^{-1}).
\end{equation}
If the parameters satisfy the condition $-1<\Re(c-a-b)\leq 0$, then the series converges for $|z|\leq1$, except at the point $z=1$.
If $\Re(c-a-b)>0$, the series extends continuously also at $z=1$. If, for simplicity, we assume $\Re(a)>0$, $\Re(b)>0$, $\Re(c-a)>0$ and $\Re(c-b)>0$, then 
we have the following three regimes for the behaviour of the series at $z=1$.
When $\Re(c-a-b)>0$,
\begin{equation}\label{limz1}
\lim_{z\to 1^-}{}_2F_1(a,b;c;z)={}_2F_1(a,b;c;1)=\frac{\Gamma(c)\Gamma(c-a-b)}{\Gamma(c-a)\Gamma(c-b)}.
\end{equation}
If $c=a+b$ and $c\not\in\Z$,
\begin{equation}\label{limz2} 
\lim_{z\to 1^-}\frac{{}_2F_1(a,b;a+b;z)}{-\log(1-z)}=\frac{\Gamma(a+b)}{\Gamma(a)\Gamma(b)}.
\end{equation}
Finally, if $\Re(c-a-b)<0$,
\begin{equation}\label{limz3}
\lim_{z\to 1^-}\frac{{}_2F_1(a,b;c;z)}{(1-z)^{c-a-b}}=\frac{\Gamma(c)\Gamma(a+b-c)}{\Gamma(a)\Gamma(b)}.
\end{equation}
For more details we refer to  
\cite[Chapter 9]{Lebedev} and \cite[Chapter 2]{ErdHTFI}.

\smallskip

We recall the following result from \cite{FrankMatzke2023}.

\begin{lemma}\label{hyper}
	Let $a\geq 0$, $b\in\R$, and $c>0$. If $c\geq\max\{a,b\}$, then the function
	$$
	[1,\infty) \ni z \mapsto z^{-a} {}_2 F_1(a,b;c;z^{-1}) 
	$$
	is non-negative. If $c\geq\max\{a+1,b\}$, it is non-increasing and, if $c\geq\max\{a+2,b\}$, it is convex.
\end{lemma}

\begin{proof}
	The first assertion follows from \cite[Lemma 5]{FrankMatzke2023}. The second assertion is a consequence of the same lemma together with \eqref{der:hyp2}, and the third one follows from \cite[Corollary~6]{FrankMatzke2023}.
\end{proof}

\subsection{A formula relating Bessel and hypergeometric functions} In this section we prove the following identity. 
For $\alpha\in \R, \alpha\neq\pm1$ and for $0<s<5$ we have 
\begin{align}\label{cosFIs}
&\int_0^{\infty} \frac{J_{\frac s2+1}(t)}{t^{2-\frac s2}}\,\cos(t\alpha)\, dt= 2^{\frac s2-2}\,\Gamma(\tfrac s2) \bigg((1-s\alpha^2)\chi_{(-1,1)}(\alpha) \nonumber\\
&
\qquad+ 2^{-s+1} \frac{\Gamma(s)}{\Gamma(\tfrac s2+2)\Gamma(\tfrac s2)} \cos\left(\frac{\pi s}2\right) |\alpha|^{-s}{}_2F_1
\left(\frac s2,\frac{s+1}2;\frac s2+2;\alpha^{-2}
\right)\chi_{[-1,1]^c}(\alpha)\bigg).
\end{align}
Identity \eqref{cosFIs} means that, if the function
\begin{equation}\label{Is}
I_s(t):=\frac{J_{\frac s2+1}(t)}{t^{2-\frac s2}}, \quad t \in (0,\infty),
\end{equation}
is extended evenly on $\R$, then its distributional Fourier transform, up to a multiplicative constant, is
the right-hand side of \eqref{cosFIs} for $0<s<5$.
 
Before proving \eqref{cosFIs} we introduce the following shorthand notation that will be used 
throughout the paper:
\begin{align}\label{kappa-s-0}
\tilde{f}_s(\alpha)&:=1-s\alpha^2,
\\\label{kappa-s}
\kappa_s&:= 2^{-s+1} \frac{\Gamma(s)}{\Gamma\big(\tfrac s2+2\big)\Gamma\big(\tfrac s2\big)}\cos\left(\frac{\pi s}2\right),\\\label{a-sF}
f_s(\alpha)&:= |\alpha|^{-s}{}_2F_1
\left(\frac s2,\frac{s+1}2;\frac s2+2;\alpha^{-2}\right),\quad |\alpha|>1,
\\\label{acca-s}
h(\alpha,s)&:=2^{\frac s2-2}\,\Gamma(\tfrac s2) \bigg( \tilde{f}_s(\alpha)\chi_{(-1,1)}(\alpha)+\kappa_sf_s(\alpha)\chi_{[-1,1]^c}(\alpha)\bigg).
\end{align}

In terms of the notation above, our claim can be rephrased as 
\begin{equation}\label{cosFIsbis}
\widehat{I_s}(\alpha)=\sqrt{\frac2\pi}h(\alpha,s).
\end{equation}
Note that $h(\cdot,s)\in L^1_{\text{loc}}(\R)$ for $0<s<1$, whereas $h(\cdot,s)\in L^1(\R)$ for  $1\leq s<5$.
Indeed, integrability close to $\alpha=\pm1$ follows by  \eqref{limz1} for $0<s<3$, by \eqref{limz2} for $s=3$, and by \eqref{limz3} for $3<s<5$. Integrability at infinity holds for any $s\geq 1$ (note that $h(\alpha,1)=0$ for $|\alpha|>1$, since $\kappa_1=0$). 

For the derivation of \eqref{cosFIsbis}, we first recall that by \cite[formula (13) on page 45]{Bateman} 
\begin{align}\label{cosFIs-book}
&\int_0^{\infty} t^{2\mu-1} J_{2\nu}(t)\,\cos(t\alpha)\, dt= \frac{2^{2\mu-1}\,\Gamma(\mu+\nu)}{\Gamma(1+\nu-\mu)}  \, {}_2F_1
\left(\nu+\mu,\mu-\nu;\tfrac12;\alpha^2 
\right)\,
\chi_{(-1,1)}(\alpha) \nonumber\\
&
\quad+ \frac{\Gamma(2\nu+2\mu)}{2^{2\nu} \Gamma(2\nu+1)} \cos\left((\nu+\mu)\pi\right) |\alpha|^{-2\nu-2\mu}{}_2F_1
\left(\nu+\mu,\nu+\mu+\tfrac12;2\nu+1;\alpha^{-2}
\right)\chi_{[-1,1]^c}(\alpha),
\end{align}
with parameters $\mu, \nu\in \R$ satisfying the requirement $-\nu<\mu<\frac34$.  Applying \eqref{cosFIs-book} with 
$\mu=\frac{s}4-\frac12$ and $\nu=\frac{s}4+\frac12$, and by \eqref{Is}, we obtain the identity
\begin{align}
&\int_0^{\infty} I_s(t) \cos(t\alpha)\, dt=
2^{\frac s2-2}\,\Gamma(\tfrac s2)\, {}_2F_1
\left(\frac s2,-1;\frac12;\alpha^2 
\right)\chi_{(-1,1)}(\alpha)  \nonumber\\\label{due}
&\qquad + 2^{-\frac s2-1}\frac{\Gamma(s)}{\Gamma(\tfrac s2+2)} \,\cos\left(\frac{\pi s}2\right) |\alpha|^{-s}
{}_2F_1
\left(\frac s2,\frac{s+1}2;\frac s2+2;\alpha^{-2}
\right)\chi_{[-1,1]^c}(\alpha),
\end{align}
where the condition $-\nu<\mu<\frac34$ results into $0<s<5$. 

Note that the first term in the right-hand side of \eqref{due} can be written more explicitly, since by applying \eqref{poli}--\eqref{2F1} with $b=-1$ we have that 
$$
{}_2F_1
\left(\frac s2,-1;\frac12;\alpha^2
\right) = 1- \frac{(s/2)_1}{(1/2)_1}\alpha^2=1-s\alpha^2.
$$
This proves \eqref{cosFIsbis} for $0<s<3$. Indeed, in this range $I_s\in L^1(0,\infty)$ and hence the right-hand side of 
 \eqref{due} is the Fourier transform of $I_s$, up to a multiplicative constant.  
 
 This is all we need for the proof of Theorem \ref{chara}. For the sake of completeness we proceed with the proof of  \eqref{cosFIsbis} in the range $3\le s\le5$ .

For $3\le s<5$ the function $I_s$ is not in $L^1(0,\infty)$ and  \eqref{due}  has to be (in principle) interpreted as a pointwise identity, where the integral in the left-hand side is an improper integral.

To prove \eqref{cosFIsbis} for $3\le s<5$, we need to show that
for any even function $\varphi$ in the Schwarz space $\mathcal{S}$ we have
\begin{equation*}
 \int_{\R} I_s(t)\, \widehat{\varphi} (t)\,dt= \sqrt{\frac{2}{\pi}}\int_{\R} h(\alpha,s)\,\varphi(\alpha)\,d\alpha,
\end{equation*}
which, more explicitly, is
\begin{equation*} 
 \int_{\R} I_s(t)\, \int_{\R} \varphi(\alpha)\,e^{-it\alpha}\,d\alpha\,dt= 2\int_{\R} h(\alpha,s)\,\varphi(\alpha)\,d\alpha.
\end{equation*}
By the evenness of $h$ in $\alpha$, and of $I_s$ and $\varphi$, this is equivalent to
\begin{equation}\label{cft4}
 \int_{0}^\infty \frac{J_{\frac s2+1}(t)}{t^{2-\frac s2}}\, \int_{0}^\infty \varphi(\alpha) \cos(t\alpha)\,d\alpha\,dt= \int_{0}^\infty h(\alpha,s)\,\varphi(\alpha)\,d\alpha.
\end{equation}
Now the above identity holds for $0<s<3$ by Fubini's Theorem. Moreover, both sides of the identity are analytic functions of $s$, due to the analyticity of the Gamma function, the Bessel function, and the hypergeometric function. Hence the two sides of the identity are equal for the whole interval $0<s<5$, which gives \eqref{cft4}. For $s\geq 5$ analyticity breaks down because, as already pointed out, $h(\alpha,s)$
is not integrable in $\alpha$ close to $\alpha=\pm 1$.

For $s=5$, we can still compute the Fourier transform of $I_5$, evenly extended on $\R$. We have that
\begin{equation}\label{cosFI5bis}
\widehat{I_5}(\alpha)=
\frac32 \, (1-5\alpha^2)\chi_{(-1,1)}(\alpha)  + \delta_{-1}(\alpha)+\delta_1(\alpha),
\end{equation}
where, for $z\in\R$, $\delta_z$ denotes the Dirac measure supported at $z$.

To prove \eqref{cosFI5bis}, we first use \cite[formula (13) on page 45]{Bateman} with $\mu=-1/4$ and $\nu=7/4$ and obtain
\begin{eqnarray*}
\int_0^{\infty} J_{\frac 72}(t) \, t^{-\frac32}\,\cos(t\alpha)\, dt & = & \frac{2^{-3/2}\,\Gamma(3/2)}{\Gamma(3)}  \, {}_2F_1
\left(\tfrac32,-2;\tfrac12;\alpha^2 
\right)\,
\chi_{(-1,1)}(\alpha)\\
& = & \sqrt{\frac\pi2}\frac1{2^3}(1-6\alpha^2+5\alpha^4)\,
\chi_{(-1,1)}(\alpha).
\end{eqnarray*}
Set $\tilde{I}_5(t)=J_{\frac 72}(t) \, t^{-\frac32}$. Since $\tilde{I}_5$, evenly extended on $\R$, is integrable on $\R$, the above formula yields
$$\widehat{\tilde{I}_5}(\alpha)=\frac1{2^3}(1- 6\alpha^2+5\alpha^4)\,
\chi_{(-1,1)}(\alpha).$$
The equality $I_5(t)=t^2\tilde{I}_5(t)$ implies that, in the distributional sense,
$$\widehat{I_5}(\alpha)=-\frac{d^2}{d\alpha^2}\widehat{\tilde{I}_5}(\alpha),$$
from which \eqref{cosFI5bis} follows.

\subsection{Ellipses and ellipsoids} For any $a\in \R^d$ we write $a_i=a\cdot e_i$, and $D(a)$ stands for the $d\times d$ diagonal matrix such that $(D(a))_{ii}=a_i$. Given $a\in \R^d$ with $a_i>0$, we let 
\begin{equation}\label{el}
E_0(a):= \left\{x\in \R^d:\ \sum_{i=1}^d\frac{x_i^2}{a_i^2}\le 1\right\}
\end{equation}
denote the compact set enclosed by the ellipsoid with semi-axes of length $a_i$ on the coordinate axes. Note that 
\begin{equation*}
E_0(a) = D(a)\overline{B},
\end{equation*}
where $B$ denotes the open unit ball $B_1(0)\subset\R^d$.
A general ellipsoid $E\subset \R^d$ centred at the origin can be obtained by rotating $E_0(a)$ in \eqref{el} with respect to the coordinate axes, namely as
\begin{equation}\label{ellrot}
E= R E_0(a)=R D(a) \overline{B},
\end{equation}
for some rotation $R\in SO(d)$.

\subsection{The Fourier transform of the candidate minimiser}\label{sect:Fmin}
In the isotropic case where $\Psi=1$ in \eqref{potdef}, for $d\geq 2$ and $s\in (\max\{d-4,0\},d)$, the minimiser of the energy $I$ over $\mathcal{P}(\R^d)$ is the probability measure $\mu_{\text{iso},d}$ defined as
\begin{equation}\label{Blatta}
\mu_{\text{iso},d}(x):= c_d(1-|x/r_d|^2)^{\frac{s+2-d}2}\chi_{\Rd\overline{B}}(x),
\end{equation}
where $c_d>0$ and $\Rd>0$ are explicit constants depending on $d$ and $s$, and where we identified the measure with its density $\mu_{\text{iso},d}\in L^1(\R^d)$; see \cite{FrankMatzke2023}. Note that the super-Coulombic range $s\geq d-2$ leads to a non-negative power $\frac{s+2-d}2\geq 0$ in \eqref{Blatta}, and hence to a bounded density. The density of the measure is instead unbounded in the sub-Coulombic range $s<d-2$. In particular, for $d=2$ the power of the density in \eqref{Blatta} is always positive (and hence the density is bounded).

The Fourier transform of $\mu_{\text{iso},d}$ is well known (see, for example, \cite[Appendix B.5]{Grafakos}). The restriction on the power of the quadratic function imposed in \cite{Grafakos} translates into the condition $s>d-4$ for \eqref{Blatta}. Adjusting the constants, due to the slightly different definition of the Fourier transform adopted in \cite{Grafakos}, we obtain for $s>d-4$
\begin{align*}
\widehat{\mu_{\text{iso},d}}(\xi) &=c_{s,d}\frac{J_{\frac s2+1}(\Rd|\xi|)}{\Rd^{\frac s2+1}|\xi|^{\frac s2+1}}
,
\end{align*}
where we set
\begin{equation*} 
c_{s,d}:= c_d\Rd^d 2^{\frac{s+2-d}2}\Gamma(\tfrac{s-d}2+2).
\end{equation*}
 
Let $E\subset \R^d$ be an ellipsoid of the form \eqref{ellrot}. We now define the (absolutely continuous) probability measure $\mu_E\in \mathcal{P}(\R^d)$ as 
\begin{equation}\label{Blatta-E}
\mu_{E}(x)=  \frac{c_d}{\prod_{j=1}^d (a_j/r_d)}\left(1-\left|D\big(\tfrac1a\big)R^Tx\right|^2\right)^{\frac{s+2-d}2}\chi_{ E}(x),
\end{equation}
which is the push-forward of the measure \eqref{Blatta} through the function $x\mapsto RD(a/\Rd)x$.
Here $D\big(\tfrac1a\big)$ is the diagonal matrix such that $(D(\tfrac 1a))_{ii}=1/a_i$.

Then it is easy to see that for $s>d-4$
\begin{align}\label{hatBlatta-E}
\widehat{\mu_E}(\xi) = \widehat{\mu_{\text{iso},d}}(D(a/r_d)R^T\xi)
= c_{s,d}\frac{J_{\frac s2+1}(|D(a)R^T\xi|)}{|D(a)R^T\xi|^{\frac s2+1}}.
\end{align}

\subsection{The Fourier representation of $W\ast \mu_E$ for $s\in (d-4,d)\cap(0,5]$.} 
Let $E$ be an ellipsoid of the form \eqref{ellrot}, and let $W$ be as in \eqref{potdef}. In this section we do not require any sign condition on $\Psi$.  

First of all, since the kernel $W$ is homogeneous of degree $-s$, its Fourier transform $\widehat{W}$ is homogeneous of degree
$-(d-s)$. Moreover,
the assumption $0<s<d$ provides local integrability of the kernel $W$. To compute the Fourier transform of $W$, it is convenient to write the profile $\Psi\in L^2(\S^{d-1})$ in terms of spherical harmonics, namely
$$\Psi=\sum_{n=0}^{\infty}\psi_n,$$
where each $\psi_n$ is a spherical harmonic of order $n$ on $\S^{d-1}$ (in particular, $\psi_0$ is just a constant).
Then
$$W= \sum_{n=0}^{\infty}W_n,\qquad\text{where }W_n(x)=\frac{1}{|x|^s}\,\psi_n\bigg(\frac{x}{|x|}\bigg).$$ 
By using \cite[Chapter V, Lemma~2]{S} for $n=0$ and \cite[Chapter III, Theorem~5]{S} for $n\geq 1$, we infer that, for suitable constants $b_{n,s,d}$,
\begin{equation}\label{hatW}
\widehat{W}(\xi)=\frac{1}{|\xi|^{d-s}}\sum_{n=0}^{\infty}b_{n,s,d} \,\psi_n \bigg(\frac{\xi}{|\xi|}\bigg)=\frac1{|\xi|^{d-s}}{\widehat{W}\bigg(\frac{\xi}{|\xi|}\bigg)},\quad s\in (0,d),
\end{equation}
provided the series at the right-hand side converges, for instance in $L^2(\S^{d-1})$, to a function which, with a little abuse of notation, we denote $\widehat{\Psi}(\xi/|\xi|):=\widehat{W}(\xi/|\xi|)$.
We recall that in our main theorem, Theorem~\ref{chara}, such a function $\widehat{\Psi}$ is assumed to be continuous on $\S^{d-1}$.
Since $\Psi$ is even, we infer that also $\widehat{\Psi}$ is even.
Finally, we set
$b_{s,d}:=b_{0,s,d}$ and we note that it is a positive constant.

Our goal is to derive a Fourier representation formula for $W\ast \mu_E$. To do so we first deal with the case $s\in (d-4,d)\cap(0,3)$, in which the integrability condition for Fourier inversion formula holds. We then extend the representation formula to the remaining  range $s\in(d-4,d)\cap[3,5]$ by analyticity.  

\subsubsection{The case of $s\in (d-4,d)\cap(0,3)$} 
First of all, by \eqref{hatBlatta-E} and \eqref{hatW} we have that for $s\in (d-4,d)\cap(0,3)$
\begin{eqnarray*}
\widehat{W\ast\mu_E}(\xi) & = & (2\pi)^{d/2}\widehat{W}(\xi) \widehat{\mu_E}(\xi)
\\
& = &  (2\pi)^{d/2} c_{s,d}
\frac{J_{\frac s2+1}(|D(a)R^T\xi|)}{|D(a)R^T\xi|^{\frac s2+1}}
\frac1{|\xi|^{d-s}}\widehat{\Psi}(\xi/|\xi|)\in L^1(\mathbb{R}^d).
\end{eqnarray*}
Integrability in $\mathbb{R}^d$ of the function at the right-hand side, for $0<s<3$, follows immediately from the asymptotic formulas for Bessel functions in Section~\ref{Besselsec}.

Hence, for $s\in (d-4,d)\cap(0,3)$ the inversion formula holds, that is,
\begin{equation}\label{inv-WmE}
(W\ast\mu_E)(x)=\int_{\R^d}\widehat W(\xi)\widehat{\mu_E}(\xi)e^{ix\cdot\xi}\, d\xi=\int_{\R^d}\widehat W(\xi)\widehat{\mu_E}(\xi)\cos(x\cdot\xi)\, d\xi
\end{equation}
for every $x\in\R^d$. Writing this integral in spherical coordinates, we obtain
\begin{align*}
&(W\ast\mu_E)(x) \\
&=c_{s,d}  \int_{\mathbb{S}^{d-1}}
\int_0^{\infty} \frac{J_{\frac s2+1}( \rho |D(a)R^T\omega|)}{\rho^{2-\frac s2}}\,\frac{\widehat{\Psi}(\omega)}{{|D(a)R^T\omega|^{\frac s2+1}}}\,\cos(x\cdot \rho\omega)\, d\rho
\,d\mathcal{H}^{d-1}(\omega).
\end{align*}
Setting
\begin{equation}\label{alpha-rho}
t:= \rho |D(a)R^T\omega|\quad\text{and}\quad \alpha(x,\omega):=\frac{x\cdot \omega}{|D(a)R^T\omega|},
\end{equation} 
we have for every $x\in\R^d$
$$
(W\ast\mu_E)(x)=c_{s,d}  \int_{\mathbb{S}^{d-1}}\left(
\int_0^{\infty} \frac{J_{\frac s2+1}(t)}{t^{2-\frac s2}}\,\cos(t\alpha(x,\omega))\, dt\right)\frac{\widehat{\Psi}(\omega)}{{|D(a)R^T\omega|^s}}
\,d\mathcal{H}^{d-1}(\omega).
$$
By using formula \eqref{cosFIs} for the radial integral, and by 
setting $\tilde c_{s,d}:= c_{s,d} \,2^{\frac s2-2}\,\Gamma(\tfrac s2)>0$ and 
\begin{equation}\label{def-g}
g_s(\omega):= \tilde c_{s,d}\frac{\widehat{\Psi}(\omega)}{{|D(a)R^T\omega|^s}},
\end{equation}
we have that for every $x\in \R^d$, using \eqref{kappa-s-0}--\eqref{a-sF},
\begin{align}\label{potential}\nonumber
(W\ast \mu_E)(x) &=\int_{\mathbb{S}^{d-1}}g_s(\omega)\bigg(\tilde{f}_s\big(\alpha(x,\omega)\big)\chi_{(-1,1)}\big(\alpha(x,\omega)\big)\\&\qquad\qquad +\kappa_s f_s\big(\alpha(x,\omega)\big)
\chi_{[-1,1]^c}\big(\alpha(x,\omega)\big)\bigg) 
\,d\mathcal{H}^{d-1}(\omega).
\end{align}
This is the representation we were looking for.

Thanks to \eqref{def-g}--\eqref{potential} we can show that $W\ast\mu_E$ is a quadratic polynomial in $ E$. Indeed,  for $x$ in the interior of $ E$ one has that
$|\alpha(x,\omega)|<1$ for any $\omega\in\mathbb{S}^{d-1}$.  
Hence, if $x$ is in the interior of $ E$, we have that 
\begin{equation}\label{pot-inside}
(W\ast\mu_E)(x)=\tilde c_{s,d}\int_{\mathbb{S}^{d-1}}\frac{\widehat{\Psi}(\omega)}{{|D(a)R^T\omega|^s}}(1-s\alpha(x,\omega)^2)
\,d\mathcal{H}^{d-1}(\omega),
\end{equation}
which is quadratic in $x$, up to an additive constant.

\subsubsection{The case of $s\in (d-4,d)\cap[3,5]$}\label{sects35} 
Here we cannot apply the inversion formula \eqref{inv-WmE} directly, since $\widehat{W\ast\mu_E}\notin L^1(\R^d).$ To deal with the non-integrable blow-up of the Fourier transform of the potential at infinity one could resort to a regularisation process, but it is  more direct to use analyticity. 

We claim that \eqref{potential} also holds for $3\le s< 5.$ Note that both sides of the equation are well defined in this range. By the special form of he Barenblatt density the convolution in the left hand side is a continuous function of $x$ on $\mathbb{R}^d.$ 
 The first term in the right hand side has a bounded integrand and the second term is an absolutely convergent integral since
$$
|f_s(\alpha(x,\omega))| \le C \, 1/(|\alpha(x,\omega)|-1)^p,\quad \text{with}\quad  p=(s-3)/2 < 1, \quad\text{provided}\quad s<5.
$$
This follows from \eqref{limz3}. Then both sides of \eqref{potential} are analytic functions of $s$ on the strip $\{s\in \mathbb{C}: 0<\Re(s)<5 \}$ by Morera's Theorem. These functions are equal on $(0,3)$  and thus they are equal on $(0,5).$

Setting
\begin{equation*}\label{}
G_s(x)= \int_{\mathbb{S}^{d-1}} g_s(\omega)\, \kappa_s f_s\big(\alpha(x,\omega)\big)
\chi_{[-1,1]^c}\big(\alpha(x,\omega)\big) 
\,d\mathcal{H}^{d-1}(\omega), \quad x\in \mathbb{R}^d,
\end{equation*}
\eqref{potential} 
becomes, for $3\le s <5,$
\begin{align}\label{potentialG}
(W\ast \mu_E)(x) &=\int_{\mathbb{S}^{d-1}}g_s(\omega)\big(1-s \alpha(x,\omega)^2\big) \chi_{(-1,1)}\big(\alpha(x,\omega)\big)+ G_s(x), \quad x\in \mathbb{R}^d,
\end{align}
with $G_s$ a non-negative function which vanishes on $E$  ($G_s$ is non-negative because $\kappa_s$ and $f_s$ are). As we will show in the next section Theorem \ref{chara} follows readily from the representation above for $3\leq s <5.$
%Indeed \eqref{potentialG} yields
%\begin{equation*}\label{}
%(W\ast \mu_E)(x) =\int_{\mathbb{S}^{d-1}}g_s(\omega)\big(1-s \alpha(x,\omega)^2\big)\,d\mathcal{H}^{d-1}(\omega), \quad x \in E,
%\end{equation*}
%which is a quadratic function of $x.$ Then the continuity argument applies and we find an ellipsoid such that EL1 is satisfied. That EL2 follows from EL1 is immediate for $3\le s < 5$ by the usual argument because $G_s$ is non-negative and can be ignored.

For $s=5$ one would like to take limit in \eqref{potentialG} as $s \rightarrow 5^-.$  There is no trouble in taking limit for the left hand side, where the dependence in $s$ is in the kernel \eqref{potdef} and in the density \eqref{thm:min}. The limit of the first term in the right hand side also exists by dominated convergence. Thus one can take the limit

\begin{equation*}\label{}
\lim_{s \to 5^-} G_s(x):=G_5(x), \quad x\in \mathbb{R}^d.
\end{equation*}
Therefore for the kernel $W$ and the measure $\mu_E$ corresponding to $s=5$ we get
\begin{align}\label{pot5}
(W\ast \mu_E)(x) =\int_{\mathbb{S}^{d-1}} g_5(\omega) \big(1-5\alpha(x,\omega)^2\big)\chi_{(-1,1)}\big(\alpha(x,\omega)\big)\,d\mathcal{H}^{d-1}(\omega)+G_5(x), \quad x\in \mathbb{R}^{d},
\end{align}
with $G_5$ a non-negative function vanishing on $E.$ As in the case $3\le s <5$, Theorem \ref{chara} follows readily from this representation for $s=5.$

Note that \eqref{potentialG} for  $3\le s \le 5$ entails that the interaction potential $W\ast \mu_E$ is a quadratic polynomial on the interior of $E.$ This follows from the fact that if $x$ is in the interior of $E$ then $|\alpha(x,\omega)|< 1, \; \omega \in \mathbb{S}^{d-1},$ which gives \eqref{pot-inside} in view of \eqref{potentialG} and \eqref{pot5}.

\section{Proof of the main result}\label{sec:main-res}
In this section we prove the core of Theorem~\ref{chara}, namely the characterisation of the minimiser of the energy $I$, whose existence and uniqueness has been established in Proposition~\ref{ex+uniq}, in terms of the Barenblatt profile on an ellipsoid.

\smallskip

Let $E$ be an ellipsoid of the form \eqref{ellrot} and let $\mu_E$ be as in \eqref{Blatta-E}. For any $x\in \mathbb{R}^d$ we define the \emph{potential} \begin{equation}\label{potential_EL}
P_E(x):=(W\ast\mu_E)(x)+\frac{|x|^2}2. 
\end{equation}
We need to show that there exists an ellipsoid $E$
such that the corresponding function $P_E$ defined as in \eqref{potential_EL} satisfies \eqref{EL1} and \eqref{EL2}. We present the proof in subsections \ref{s:EL1} and  \ref{s:EL2}, devoted to the first and the second Euler-Lagrange condition, respectively.

\subsection{The first Euler-Lagrange condition}\label{s:EL1}

We emphasise that in this subsection we will make use of the strict positivity condition $\widehat{W}> 0$ on $\mathbb{S}^{d-1}$.  

%We proceed differently for $s\in [d-3,d)\cap(0,3)$ and $s\in [d-3,d)\cap[3,5]$, since in the latter case we only have a representation of the regularised potential, and hence more care will be needed.

By the regularity and evenness of the potential $P_E$ in \eqref{potential_EL}, proving condition \eqref{EL1} is equivalent to showing that the Hessian of $P_E$ vanishes on $ E$. 
By \eqref{alpha-rho} and \eqref{pot-inside}, which holds for $0<s \le5,$ this is equivalent to showing that for $i,j=1,\dots,d$
\begin{equation}\label{EL1-Fourier}
\gamma_{s,d} \int_{\mathbb{S}^{d-1}}\frac{\omega_i\omega_j \widehat{\Psi}(\omega)}{|D(a)R^T\omega|^{s+2}}\,d\mathcal{H}^{d-1}(\omega)=\delta_{ij},
\end{equation}
$\delta_{ij}$ being the Kronecker delta and $\gamma_{s,d}:=2s \,\tilde{c}_{s,d}$.

We need to show that
 there exist $a=(a_1,\dots,a_d)\in \R^d$, with $a_i>0$, and $R\in SO(d)$ such that \eqref{EL1-Fourier} is satisfied.

We note that 
$$
|D(a)R^T\omega|=\left(RD(a^2)R^T\omega\cdot \omega\right)^{1/2}=: (M\omega\cdot \omega)^{1/2},
$$
where $M\in \mathbb{M}_+$, and $\mathbb{M}_+$ denotes the space of symmetric and positive definite matrices in $\R^{d\times d}$. By symmetry, we can consider $\mathbb{M}_+$ as an open subset of $\R^{\frac{d(d+1)}{2}}$. Then solving \eqref{EL1-Fourier} is equivalent to finding $M\in \mathbb{M}_+$ such that
\begin{equation}\label{EL1-M}
\gamma_{s,d}
\int_{\mathbb{S}^{d-1}}\frac{\omega_i\omega_j \widehat{\Psi}(\omega)}{(M\omega\cdot\omega)^{\frac s2+1}}\,d\mathcal{H}^{d-1}(\omega)=\delta_{ij}.
\end{equation}

\noindent  We prove \eqref{EL1-M} via a continuity argument. Let $t\in [0,1]$; we define the potential 
$$
W_t(x):= \frac1{|x|^s}\left(t\Psi\left(\frac{x}{|x|}\right)+\frac{1-t}{b_{s,d}}\right),\quad x\in \R^d,\ x\neq 0,
$$
with $b_{s,d}:= b_{0,s,d}>0$ defined in \eqref{hatW}, namely the Fourier transform of $x\mapsto \frac1{|x|^s}$ is $\xi\mapsto b_{s,d}\, \frac1{|\xi|^{d-s}}$.
By  \eqref{hatW} we have
\begin{equation}\label{hatWt}
\widehat{W_t}(\xi)=\frac1{|\xi|^{d-s}}\widehat{\Psi_t}(\xi/|\xi|):=\frac1{|\xi|^{d-s}} \big( t \widehat{\Psi}(\xi/|\xi|) +1-t\big),
\end{equation}
hence the assumption $\widehat{W}>0$ on $\mathbb{S}^{d-1}$ implies that  
$\widehat{W_t} >0$ on $\mathbb{S}^{d-1}$ for all $t\in [0,1]$.
Now, we define the function $L: [0,1]\times \mathbb{M}_+ \to \R^{\frac{d(d+1)}{2}}$ as
\begin{align*}
L_{ij}(t,M):=&\gamma_{s,d}
\int_{\mathbb{S}^{d-1}}\frac{2\omega_i\omega_j \widehat{\Psi_t}(\omega)}{(M\omega\cdot\omega)^{\frac s2+1}}\,d\mathcal{H}^{d-1}(\omega) \quad\text{for }i<j,\\
L_{ii}(t,M):=&\gamma_{s,d}
\int_{\mathbb{S}^{d-1}}\frac{\omega_i^2 \widehat{\Psi_t}(\omega)}{(M\omega\cdot\omega)^{\frac s2+1}}\,d\mathcal{H}^{d-1}(\omega)-1,\\
\end{align*}
where we have used that, by symmetry, we can restrict to $i\leq j$. For fixed $t$ consider the equation 
\begin{equation}\label{LtM}
L(t,M)= 0,
\end{equation}
in the unknown $M\in \mathbb{M}_+$. Clearly, \eqref{EL1-M} is equivalent to \eqref{LtM} for $t=1$.

We define a subset $T\subseteq [0,1]$ as 
$$
T:=\left\{
t\in [0,1]: \text{there exists } M\in \mathbb{M}_+ \, \text{such that } L(t,M)=0 
\right\}.
$$
We observe that $T\neq \emptyset$, since $0\in T$. Indeed, the case $t=0$ corresponds to isotropic Riesz interactions, for which $M$ is a multiple of the identity. 

Our claim follows if we show that $1\in T$. To that aim, we will prove that $T$ is both open and closed in $[0,1]$, from which it follows that $T=[0,1].$

To show that $T$ is open in $[0,1]$, let $t_0\in T$; hence there exists $M_0\in \mathbb{M}_+$ such that $L(t_0,M_0)=0$. It is not difficult to prove that, for every $M\in \mathbb{M}_+$ and every $t\in [0,1]$, the differential of $L(t,\cdot)$ at $M$, namely the linear operator
$$
\frac{\partial L}{\partial M}(t,M): \R^{\frac{d(d+1)}{2}} \to \R^{\frac{d(d+1)}{2}}
$$
is invertible. To show this we set $B:=\frac{\partial L}{\partial M}(t,M)$ and we calculate its matrix entries.
For any $i\le j$ and $k\le l$ we define
$$A_{(i,j),(k,l)}:=\gamma_{s,d} \Big(\frac s2+1\Big) 
\int_{\mathbb{S}^{d-1}}\frac{\omega_i\omega_j\omega_k\omega_l \widehat{\Psi_t}(\omega)}{(M\omega\cdot\omega)^{\frac s2+2}}\,d\mathcal{H}^{d-1}(\omega).$$
We obtain 
$$
B_{(i,j),(k,l)}=\frac{\partial L_{ij}}{\partial M_{kl}}(t,M) = \begin{cases}
-  A_{(i,i),(k,k)} &\text{ for $i= j$ and $k=l$},\\ 
- 2A_{(i,i),(k,l)} &\text{ for $i=j$ and $k<l$},\\
- 2A_{(i,j),(k,k)} & \text{ for $i< j$ and $k=l$},\\
 - 4A_{(i,j),(k,l)} & \text{ for $i<j$ and $k<l$}.
\end{cases}
$$
The invertibility of $B$ follows from the fact that it is a negative-definite matrix. Indeed, if  $N\in \R^{\frac{d(d+1)}{2}}$ and we identify $N$ as a symmetric matrix in $\R^{d\times d}$, then
$$
B N\cdot N= -\gamma_{s,d} \Big(\frac s2+1\Big) 
\int_{\mathbb{S}^{d-1}}\frac{(N\omega\cdot\omega)^2 \widehat{\Psi_t}(\omega)}{(M\omega\cdot\omega)^{\frac s2+2}}\,d\mathcal{H}^{d-1}(\omega).
$$
The quantity above is always non-positive and it is equal to $0$ if and only if $N\omega\cdot\omega=0$ $\mathcal{H}^{d-1}$-a.e.\ on $\mathbb{S}^{d-1}$, that is, for $N=0$.

Since $\frac{\partial L}{\partial M}(t_0,M_0)$ is invertible, by the Implicit Function Theorem there exists an open set $U\subset\R$ with $t_0\in U$, and a function 
$$
\mathcal{M}: U\cap [0,1]\to \mathbb{M}_+
$$
such that $\mathcal M(t_0)=M_0$, and  $L(t,\mathcal M(t))=0$ for every $t\in U\cap [0,1]$. Hence we have found an open set $U$ such that $t_0\in U\cap[0,1] \subset T$. This proves that $T$ is open in $[0,1]$.

Finally, we prove that $T$ is closed. To this aim, let $(t_n)\subset T$ be a sequence converging to $t_0\in [0,1]$. We claim that $t_0\in T$.

First of all, since $(t_n)\subset T$, for every $n\in \mathbb{N}$ there exists $M_n\in \mathbb{M}_+$ satisfying $L(t_n, M_n)=0$, namely 
\begin{equation}\label{qqq}
a_{ij}:=\gamma_{s,d}
\int_{\mathbb{S}^{d-1}}\frac{\omega_i\omega_j \widehat{\Psi_{t_n}}(\omega)}{(M_n\omega\cdot\omega)^{\frac s2+1}}\,d\mathcal{H}^{d-1}(\omega)=\delta_{ij}.
\end{equation}
By adding the diagonal terms, we obtain the equality
\begin{equation}\label{eee}
\sum_{i=1}^da_{ii}=
\gamma_{s,d}
\int_{\mathbb{S}^{d-1}}\frac{\widehat{\Psi_{t_n}}(\omega)}{(M_n\omega\cdot\omega)^{\frac s2+1}}\,d\mathcal{H}^{d-1}(\omega)=d,
\end{equation}
whereas we have
\begin{equation}\label{eee2}
{\rm tr}\, M_n=\sum_{i,j=1}^d(M_n)_{ij}a_{ij}=
\gamma_{s,d}
\int_{\mathbb{S}^{d-1}}\frac{\widehat{\Psi_{t_n}}(\omega)}{(M_n\omega\cdot\omega)^{\frac s2}}\,d\mathcal{H}^{d-1}(\omega),
\end{equation}
where $\text{tr}\, M$ denotes the trace of the matrix $M.$
By H\"older's inequality and by \eqref{eee} we deduce that
\begin{eqnarray}\nonumber
{\rm tr}\, M_n & \leq & \gamma_{s,d}
\Big(\int_{\mathbb{S}^{d-1}}\frac{\widehat{\Psi_{t_n}}(\omega)}{(M_n\omega\cdot\omega)^{\frac s2+1}}\,d\mathcal{H}^{d-1}(\omega)\Big)^{\frac{s}{s+2}}
\Big(\int_{\mathbb{S}^{d-1}}\widehat{\Psi_{t_n}}(\omega)\,d\mathcal{H}^{d-1}(\omega)\Big)^{\frac2{s+2}}
\\
& = & \gamma_{s,d}^{\frac2{s+2}} d^{\frac{s}{s+2}}\Big(\int_{\mathbb{S}^{d-1}}\widehat{\Psi_{t_n}}(\omega)\,d\mathcal{H}^{d-1}(\omega)\Big)^{\frac2{s+2}}. \label{trace}
\end{eqnarray}
Since $\widehat{W}$ is by assumption continuous and strictly positive on $\mathbb{S}^{d-1}$, there exist $C_0, C_1>0$ such that $C_0\leq \widehat{W}=\widehat{\Psi}\leq C_1$
on $\mathbb{S}^{d-1}$, hence \eqref{hatWt} gives the bound 
\begin{equation}\label{hatWt-2}
\tilde C_0:=\min\{C_0,1\}\leq \widehat{\Psi_{t_n}} = t_n\,\widehat{\Psi}+(1-t_n)\leq \max\{C_1,1\},
\end{equation}
for every $n\in\mathbb N$. Therefore, the right-hand side of \eqref{trace} is uniformly bounded with respect to $n$.
Recall that if $M=(m_{ij})\in \mathbb{M}_+$, then $|m_{ij}|\le \text{tr}(M), \; \text{for all}\; i,j$. Hence
by compactness, up to subsequences $M_n\to M_0$, where $M_0$ is a positive semi-definite matrix. 
We now show that in fact $M_0$ is positive definite.

By letting $n\to +\infty$ in \eqref{eee}, and by Fatou's Lemma we have 
\begin{equation}\label{eee0}
\gamma_{s,d}
\int_{\mathbb{S}^{d-1}}\frac{\widehat{\Psi_{t_0}}(\omega)}{(M_0\omega\cdot\omega)^{\frac s2+1}}\,d\mathcal{H}^{d-1}(\omega)\leq d.
\end{equation}
Thus, $M_0$ cannot be identically zero.
Assume now, for contradiction, that $M_0$ is not positive definite. With no loss of generality we can assume that the $d$-th coordinate vector is an eigenvector for $M_0$ with eigenvalue $0$.
Then, for every $\omega\in \mathbb{S}^{d-1}$ 
$$
M_0\omega\cdot \omega\leq \|M_0\|_\infty \sum_{i=1}^{d-1}\omega_i^2, \qquad \|M_0\|_\infty :=\max_{i,j}(M_0)_{ij}>0.
$$
From \eqref{hatWt-2} and \eqref{eee0} we then obtain the following bound
$$
\int_{\mathbb{S}^{d-1}}\frac{1}{( \sum_{i=1}^{d-1}\omega_i^2)^{\frac s2+1}}\,d\mathcal{H}^{d-1}(\omega)\leq d\, \|M_0\|_\infty^{\frac s2+1}\,(\tilde C_0 \gamma_{s,d})^{-1},
$$ 
which leads to a contradiction, since the integral on the left-hand side diverges for $s\geq d-3$. We emphasize that this is the only point where we need the restriction $d-3\le s.$ Hence $M_0$ is positive definite. Passing to the limit in \eqref{qqq}, by the Dominated Convergence Theorem, we have 
\begin{equation*} 
\gamma_{s,d}
\int_{\mathbb{S}^{d-1}}\frac{\omega_i\omega_j \widehat{\Psi_{t_0}}(\omega)}{(M_0\omega\cdot\omega)^{\frac s2+1}}\,d\mathcal{H}^{d-1}(\omega)=\delta_{ij}.
\end{equation*}
This proves that $t_0\in T$, and that $T$ is closed, and concludes the proof of the claim.

%%%%%%%%%%%%%%%%%%%%%%%%%%%%%%%%%%%

\subsection{The second Euler-Lagrange condition \eqref{EL2}} \label{s:EL2}

In this subsection we show that the first Euler-Lagrange condition \eqref{EL1} implies the second one \eqref{EL2}. This part of the proof only requires $\widehat{W}\geq 0$ on $\mathbb{S}^{d-1}$, and $s> d-4$ rather than $s\geq d-3$, see Remark \ref{remark:stronger}.

Our argument is  considerably more involved in the case $s\in [d-3,d)\cap(0,3)$ than in the case  $s\in[d-3,d)\cap[3,5].$

\subsubsection{The second Euler-Lagrange condition for $s\in [d-3,d)\cap(0,3)$}

Let $E$ be an ellipsoid such that the corresponding measure $\mu_E$ satisfies the first Euler-Lagrange condition, namely let $C\in \R$ be such that for every $x\in  E$ 
\begin{align}\label{elpri}
C=P_E(x) = (W\ast\mu_E)(x)+\frac{|x|^2}2
=\int_{\mathbb{S}^{d-1}} g_s(\omega)(1-s\alpha^2(x,\omega))\,d\mathcal{H}^{d-1}(\omega)+\frac{|x|^2}2,
\end{align}
where the function $g_s$ is defined in \eqref{def-g}. 
Since $0\in E$, we obtain the conditions 
\begin{equation}\label{bla}
\begin{cases}
\displaystyle
C=P_E(0)=\int_{\mathbb{S}^{d-1}} g_s(\omega)\,d\mathcal{H}^{d-1}(\omega),\\
\displaystyle 
\frac{|x|^2}2=s \int_{\mathbb{S}^{d-1}} g_s(\omega)\alpha^2(x,\omega)\,d\mathcal{H}^{d-1}(\omega)
\quad \text{for every } x\in  E.
\end{cases} 
\end{equation}
Note that the second identity in \eqref{bla} holds for each $x\in 	\mathbb{R}^d,$ since the two members are quadratic polynomials in $x$ that coincide on $ E$. By \eqref{potential} and \eqref{bla}, using \eqref{kappa-s-0}--\eqref{a-sF}, we have, for each $x\in E^c,$
\begin{equation}\label{potex}
\begin{split}
P_E(x) = & \int_{\mathbb{S}^{d-1}} g_s(\omega)(1-s\,\alpha^2(x,\omega))\,d\mathcal{H}^{d-1}(\omega)+\frac{|x|^2}2\\*[7pt]
&+  \int_{\mathbb{S}^{d-1}} g_s(\omega)\chi_{[-1,1]^c}(\alpha(x,\omega))\bigg(
s\alpha^2(x,\omega)-1+ \kappa_s f_s(\alpha(x,\omega))
\bigg)\,d\mathcal{H}^{d-1}(\omega)\\*[7pt]
=& \ P_E(0) +  \int_{\mathbb{S}^{d-1}} g_s(\omega)\chi_{[-1,1]^c}(\alpha(x,\omega))\bigg(
s\alpha^2(x,\omega)-1+ \kappa_s f_s(\alpha(x,\omega))
\bigg)\,d\mathcal{H}^{d-1}(\omega).
\end{split}
\end{equation}
It is convenient to write $\kappa_s=K(s)\cos\left(\tfrac{\pi s}{2}\right)$, 
where $K(s):=\frac{2^{-s+1}\Gamma(s)}{\Gamma\big(\tfrac s2+2\big)\Gamma\big(\tfrac s2\big)}$. 

We claim that for $|\alpha|>1$
\begin{equation}\label{claimEL2}
F(\alpha,s):=s\alpha^2-1+K(s) \cos\left(\frac{\pi s}2\right) f_s(\alpha)
\geq 0.
\end{equation}
Since $g_s\geq 0$, this inequality implies the second Euler-Lagrange condition for $s\in [d-3,d)\cap (0,3)$. Inequality \eqref{claimEL2} is clearly true for $s=1$ (which belongs to the range of $s$ considered here if $d<5$), since $F(\alpha,1)=\alpha^2-1$. 
So, in what follows we implicitely assume $s\neq1$. Note that, since $F$ is even in the variable $\alpha$, it is sufficient to prove \eqref{claimEL2} for $\alpha>1$.

To prove \eqref{claimEL2}, we introduce the function
$$
\phi(z):= z^{-\frac{s}2} \, {}_2F_1\left(\frac s2,\frac{s+1}2;\frac s2+2;z^{-1}\right)
\qquad\text{for}\ z\in (1,\infty)
$$
(for simplicity, the dependence on $s$ is not reflected in the notation), and rewrite  \eqref{claimEL2} in terms of $\phi$. According to \eqref{limz1}, $\phi$ extends continuously to $z=1$ and
$$
\phi(1)=\frac{\Gamma(\tfrac s2+2)\Gamma(\tfrac{3-s}2)}{\Gamma(2)\Gamma(\tfrac 32)} =\frac{2}{\sqrt \pi} \Gamma(\tfrac s2+2)\Gamma(\tfrac{3-s}2),
$$
since $\Gamma(2)=1$ and $\Gamma(\tfrac32)=\frac{\sqrt{\pi}}2$. By properties (2) and (4) of the Gamma function recalled in Section~\ref{sect:propG} we can write 
\begin{equation}\label{gamma-gamma}
\Gamma(\tfrac{3-s}2)=(\tfrac{1-s}2)\Gamma(\tfrac{1-s}2), 
 \qquad \frac{\Gamma(s)}{\Gamma(\tfrac s2)}=\frac{2^{s-1}}{\sqrt \pi}{\Gamma(\tfrac{s+1}2)},
\end{equation}
so that 
$$
K(s)\phi(1) = \frac1\pi(1-s)\Gamma(\tfrac{s+1}2)\Gamma(\tfrac{1-s}2) .
$$
Since by property (3) in Section~\ref{sect:propG} we have
\begin{equation}\label{gamma-gamma2}
\Gamma(\tfrac{s+1}2)\Gamma(\tfrac{1-s}2)=\frac{\pi}{\sin(\tfrac{\pi}2(s+1))} = \frac{\pi}{\cos(\tfrac{\pi s}2)},
\end{equation}
we deduce that 
\begin{equation}
	\label{eq:constant}
	K(s) \cos\left(\frac{\pi s}2\right) = \frac{1-s}{\phi(1)}.
\end{equation}
Consequently, writing $z$ instead of $\alpha^2$, the claimed inequality \eqref{claimEL2} can be written as
\begin{equation}
	\label{eq:claimEL2alt}
	s z - 1 + \frac{1-s}{\phi(1)} \, \phi(z) \geq 0
	\qquad\text{for all}\ z> 1.
\end{equation}
To prove it, we distinguish the cases $s<1$ and $s>1$.

\subsubsection*{Case $s\in[d-3,d)\cap(1,3)$}
By Lemma~\ref{hyper} the function $\phi$ is non-increasing. Therefore, $\phi(z)\leq \phi(1)$ for all $z\geq 1$. 
Since $\cos(\pi s/2)<0$ and $K(s)>0$, by \eqref{eq:constant} we have $(1-s)/\phi(1)<0$, hence
$$
s z - 1 + \frac{1-s}{\phi(1)} \phi(z) \geq sz - 1 + \frac{1-s}{\phi(1)} \phi(1) = s(z-1) \geq 0
\qquad\text{for all}\ z\geq 1,
$$
which proves \eqref{eq:claimEL2alt}.

\subsubsection*{Case $s\in[d-3,d)\cap(0,1)$}
By Lemma \ref{hyper} the function $\phi$ is convex. By \eqref{der:hyp2}, \eqref{limz1}, and the properties of the Gamma function recalled in Section~\ref{sect:propG}, 
$\phi'$ extends continuously to $z=1$ and 
$$
\phi'(1) = - \frac{s}{1-s}\phi(1).
$$
By convexity $\phi(z) \geq \phi(1) + \phi'(1)(z-1)$ for all $z\geq 1$. 
Since $\cos(\pi s/2)>0$ and $K(s)>0$, by \eqref{eq:constant} we have $(1-s)/\phi(1)>0$, hence 
$$
s z - 1 + \frac{1-s}{\phi(1)} \phi(z) \geq sz - 1 + \frac{1-s}{\phi(1)}\left( \phi(1) + \phi'(1)(z-1)\right) = \left( s + \frac{1-s}{\phi(1)}\,\phi'(1) \right)(z-1)=0
$$
for all $z\geq1$, which proves \eqref{eq:claimEL2alt}.\medskip

This concludes the proof of \eqref{claimEL2} and therefore of the second Euler-Lagrange equation for $s\in [d-3,d)\cap (0,3)$.

\subsubsection{The second Euler-Lagrange condition for $s\in [d-3,d)\cap[3,5]$}\label{sub322}
Let $E$ be an ellipsoid such that the corresponding measure $\mu_E$ satisfies the first Euler-Lagrange condition 
\eqref{elpri}. As in the previous subsection, using \eqref{elpri}, \eqref{potentialG} and \eqref{pot5}, we get 
\begin{equation*}\label{}
\begin{split}
&P_E(x)  \\*[7pt]
& =\ P_E(0) +  \int_{\mathbb{S}^{d-1}} g_s(\omega)\big(
s\alpha^2(x,\omega)-1\big) \chi_{[-1,1]^c}(\alpha(x,\omega))\,d\mathcal{H}^{d-1}(\omega) +G_s(x).
\end{split}
\end{equation*}
Since the second and third terms in the right hand side are non-negative the second Euler-Lagrange condition follows.

\begin{remark}\label{remark:stronger}
In subsection \ref{s:EL2} we have shown that the first Euler-Lagrange condition implies the second one. This argument uses the assumption $s>d-4$ (in order for the measure $\mu_E$ to be well-defined), but it does not rely on the assumption $s\geq d-3$ (which was used to find an ellipsoid for which the first Euler-Lagrange condition is satisfied). We also note that this step of the proof only requires $\widehat{W}\geq 0$ on $\mathbb{S}^{d-1}$.
\end{remark}

%%%%%%%%%%%%%%%%%%%%%%%%%%%%%%%%%%%%%%%%%%%%
%%%%%%%%%%%%%%%%%%%%%%%%%%%%%%%%%%%%%%%%%%%%

\section{The loss of dimension in the degenerate case}\label{lossdimsec}

In this section we consider the degenerate case where the Fourier transform of the anisotropic potential is non-negative on $\mathbb{S}^{d-1}$, but not strictly positive.

More precisely, let $W_0$ be a potential satisfying the assumptions of Theorem~\ref{chara}, with $\widehat{W_0}\geq 0$ on $\mathbb{S}^{d-1}$, but not strictly positive. Let $\Psi_0$ denote its profile, as in \eqref{potdef}, and $I_0$ the corresponding energy, as in \eqref{en:general}. For $\varepsilon>0$, we `lift' the potential $W_0$ by setting, for $x\neq 0$, 
$$
W_\varepsilon(x):=\frac{1}{|x|^s}\left(\Psi_0\left(\frac{x}{|x|}\right)+ \frac{\varepsilon}{b_{s,d}}\right), 
$$
where $b_{s,d}:=b_{0,s,d}$ is defined in \eqref{hatW} so that the Fourier transform of $x\mapsto \frac{1}{|x|^s}$ is $\xi\mapsto\frac{b_{s,d}}{|\xi|^{d-s}}$. Then, since $b_{s,d}>0$ we still have that $W_\varepsilon>0$ on $\mathbb{S}^{d-1}$, and moreover 
$$
\widehat{W_\varepsilon} = \widehat{W_0}+\varepsilon>0 \qquad \text{ on } \mathbb{S}^{d-1}.
$$
Let $I_\varepsilon$ denote the energy as in \eqref{en:general}, with potential $W_\varepsilon$. By Theorem~\ref{chara}, the minimiser $\mu_\varepsilon$ of $I_{\varepsilon}$ is as in \eqref{thm:min}, with $a^\varepsilon_i>0$. 
As in the proof of Proposition~\ref{ex+uniq}, one can show that the sequence $(\mu_\varepsilon)_\varepsilon$ is tight, hence, up to subsequences, $\mu_\varepsilon$ converges in the narrow sense to some measure $\mu_0\in \mathcal{P}(\R^d)$, as $\varepsilon\to 0^+$. Moreover, it is easy to show that $\mu_0$ is the unique minimiser of $I_0$. We can characterise $\mu_0$ in terms of the limit $\bar a\in [0,+\infty)^{d}$ of $(a^\varepsilon)_\varepsilon$ as follows.

We note that $\bar a$ cannot be identically $0$, since $\mu_0$ would coincide with $\delta_0$ and $I_0(\delta_0)=W_0(0)=+\infty$, hence $\delta_0$ cannot be the minimiser of $I_0$. 
If $\bar a_i>0$ for any $i=1,\ldots,d$, then the minimiser $\mu_0$ is of the form \eqref{thm:min} and its support is fully dimensional.
Otherwise, let us 
assume that $\bar a$ has only $k\in \{1,\dots, d-1\}$ strictly positive components, and denote by $\bar a(k)\in \R^d$ a vector with the same components as $\bar a$, but rearranged so that $\bar a(k)_i>0$ for $i=1,\dots,k$, and $\bar a(k)_i=0$ for $i=k+1,\dots,d$.
Then $\mu_0(x)=\mu_{E(\bar a(k))}(R^Tx)$ for a suitable $R\in SO(d)$, where we set
\begin{equation*} 
{E}(\bar a(k)):=\left\{(x_1,\dots x_k)\in \R^k:\ \sum_{i=1}^k\frac{x_i^2}{\bar a(k)_i^2} \leq 1\right\},
\end{equation*}
and 
\begin{equation}\label{meas-k}
\mu_{E(\bar a(k))}(x)= \frac{\tilde c_{s,d,k}}{\Pi_{i=1}^k(\bar a(k)_i/r_d)}\bigg(1-\sum_{i=1}^k\frac{x_i^2}{\bar a(k)_i^2} \bigg)^{\!\frac{s+2-k}2}\!\!\!\!\!\!\!\!\chi_{ {E}(\bar a(k))}(x_1,\dots,x_k)\otimes \delta_0(x_{k+1},\dots,x_d).
\end{equation}
In \eqref{meas-k} the constant $\tilde c_{s,d,k}$ is an explicit normalisation constant and $r_d>0$ is the same as in \eqref{thm:min}.

We now discuss the possible minimisers of $I_0$ in terms of the dimension of their supports, depending on the homogeneity $s$ of the potential. 

Note that $I_0$  is bounded from below by a positive multiple of the isotropic Riesz energy $I_{\textrm{iso}}$, corresponding to $W_{\textrm{iso}}(x)=1/|x|^s$. For the isotropic Riesz energy we have  
\begin{multline*}
I_{\textrm{iso}}(\mu_{E(\bar a(k))})\\\geq \frac{\tilde c_{s,d,k}}{\Pi_{i=1}^k(\bar a(k)_i/r_d)} \int_{ {E}(\bar a(k))}\!\!\!\bigg(\int_{\R^d} \frac{1}{|x(k)-y|^s}d\mu_{E(\bar a(k))}(y)\bigg) \bigg(1-\sum_{i=1}^k\frac{x_i^2}{\bar a(k)_i^2} \bigg)^{\frac{s+2-k}2}\!\!\!\!\!\!\! dx(k).
\end{multline*}
Moreover, for any $(x_1,\dots,x_k)\in  E(\bar a(k))$
\begin{align*}
&\int_{\R^d} \frac{1}{|x(k) -y|^s}d\mu_{E(\bar a(k))}(y)\\
&\hspace{3cm}=\frac{\tilde c_{s,d,k}}{\Pi_{i=1}^k(\bar a(k)_i/r_d)} \int_{ {E}(\bar a(k))} \frac{1}{|x(k)-y(k)|^s} \bigg(1-\sum_{i=1}^k\frac{y_i^2}{\bar a(k)_i^2} \bigg)^{\frac{s+2-k}2}\!\!\!\!\!\!\! dy(k),
\end{align*}
which is equal to $+\infty$ for $s\geq k$. In other words, 
$$
I_0(\mu_{E(\bar a(k))})=+\infty \quad \text{for } s\geq k.
$$
This means that for $s\geq k$ the minimiser $\mu_0$ of $I_0$ cannot be supported on a $k$-dimensional set, and its support must be at least $(k+1)$-dimensional. In particular, the minimiser is fully dimensional, and given by \eqref{thm:min}, if $s\in [d-1,d)$. If, however, e.g., $s\in [d-2,d-1)$, then the minimiser must be supported on a set of dimension at least $d-1$, so there may or may not be a loss of dimension.

We recall that for Coulomb interactions $s=d-2$ in three dimensions, in \cite[Example 3.4]{MMRSV-3d} a potential $W_0$ is constructed such that the corresponding minimiser is supported on a two-dimensional ellipse, so the loss of dimensionality of the minimiser can in fact occur.

We also remark that the loss of dimensionality seems to be related to properties of $\widehat{W_0}$, rather than of $W_0$. It was in fact shown in \cite{Mora-Muenster}, for $s=d-2$, that if the minimiser has a $(d-1)$-dimensional support, then the normal to the hyperplane containing the support has to be a direction of degeneracy for $\widehat{W_0}$.

\bigskip

\noindent
\textbf{Acknowledgements.} 
RLF was partially support through US National Science Foundation grant DMS-1954995, as well as through the German Research Foundation through EXC-2111-390814868 and TRR 352-Project-ID 470903074. 
JM and JV have been partially supported
by 2021SGR-00071 (Catalonia) and PID2024-155320NB-I00  (Mineco, Spain).
MGM is member of GNAMPA--INdAM. 
MGM acknowledges support from PRIN 2022 (Project no. 2022J4FYNJ), funded by MUR, Italy, and the European Union -- Next Generation EU, Mission~4 Component~1 CUP~F53D23002760006.
LR is supported by the Italian MUR through the PRIN 2022 project n.2022B32J5C, under the National Recovery and Resilience Plan (PNRR), Italy, funded by the European Union  - Next Generation EU, Mission 4 Component 1 CUP~F53D23002710006, and by GNAMPA-INdAM through 2025 projects. 
LS acknowledges support by the EPSRC under the grants EP/V00204X/1 and EP/V008897/1. 
Part of this work was done during a visit of JM, MGM, LR, and JV to Heriot-Watt University, whose kind hospitality is gratefully acknowledged.

\end{document}